\documentclass[12pt]{amsart}

\usepackage{amsmath,amsthm,amscd,euscript,longtable}
\def \r{\mathbb R}

\def \({\langle}
\def \){\rangle}


\DeclareMathOperator{\il}{l\ell} 
\DeclareMathOperator{\isin}{lsin}

\newtheorem{theorem}{Theorem}[section]

\newtheorem{proposition}[theorem]{Proposition}
\newtheorem{corollary}[theorem]{Corollary}
\theoremstyle{remark}
\newtheorem{remark}[theorem]{Remark}
\theoremstyle{definition}
\newtheorem{definition}[theorem]{Definition}
\newtheorem{example}[theorem]{Example}
\newtheorem{problem}{Problem}

\title{Continued fractions and the second Kepler law}
\author{Oleg Karpenkov}
\date{9 October 2009}

\thanks{The work is partially supported by RFBR SS-709.2008.1
grant and by FWF grant No.~S09209.}

\keywords{Semidiscrete surfaces, flexibility, infinitesimal
flexibility}

\email[Oleg N. Karpenkov]{karpenkov@tugraz.at}
\address{TU Graz /Kopernikusgasse 24, A 8010 Graz, Austria/}

\begin{document}
\input{epsf}
\begin{abstract}
In this paper we introduce a link between geometry of ordinary
continued fractions and trajectories of points that moves
according to the second Kepler law. We expand geometric
interpretation of ordinary continued fractions to the case of
continued fractions with arbitrary elements.
\end{abstract}
\maketitle

\tableofcontents

\section*{Introduction}

In classical geometry of numbers the elements of an ordinary
continued fraction for a real number $\alpha\ge 1$ are obtained
from a {\it sail} (i.e. a broken line bounding the convex hull of
all points with integer coefficients in certain cone). In present
paper we find broken lines generalizing sails to the case of
continued fractions with arbitrary elements. This in its turn
leads to the definition of ``infinitesimal'' continued fractions,
whose sails would be differentiable curves. Such ``infinitesimal''
continued fractions are defined by two density functions: areal
and angular densities.

The areal density function has a remarkable physical meaning.
Consider an observer at the origin and let the body move along the
curve with the velocity inverse to the areal density function.
Then the body moves according to the second Kepler law with
respect to the observer, i.e. with constant sector area velocity.

\vspace{2mm}

This paper is organized as follows. In the first section we study
the classical case of ordinary continued fractions. In Section~2
we expand the notion of the sail to the case of continued fraction
with arbitrary elements. Further we show how to write continued
fractions starting with broken lines. We generalize the proposed
construction of sails to the case of curves, give an analog of
continued fractions, and show several examples in Section~3.

\vspace{2mm}

For a nice reference to general theory of continued fractions we
suggest the book~\cite{Khin}. Several works are devoted to
geometry of continued fractions (e.g. \cite{Irw}, \cite{Sta}) and
to their generalizations to multidimensional case (\cite{Arn2},
\cite{Kor3}, \cite{Kar1}, etc). Notice that the case of broken
lines with integer edges discussed in~\cite{KarTrig} is a
particular subcase of geometric definitions introduced in Section
2.

\vspace{2mm}

The author is grateful to V.~I.~Arnold for constant attention to
this work, and to Technische Universit\"at Graz for hospitality
and excellent working conditions.

\section{Geometry of ordinary continued fractions}

In this section we briefly introduce geometric aspects of ordinary
continued fractions.

We start with general notions of continued fractions. For
arbitrary sequence of real numbers $(a_0, a_1, \ldots)$ the {\it
continued fraction} is an expression
$$
a_0+\frac{1}{a_1+\frac{\displaystyle 1}{\displaystyle a_2+\ldots}}
$$
denoted $[a_0,a_1,a_2,\ldots]$. In case of a finite sequence we
get some real number (or sometimes $\infty$). In case of infinite
continued fraction the expression means a limit of a sequence
$([a_0,\ldots, a_n])$ while $n$ tends to infinity. Notice that
such limit does not exist for all sequences. A continued fraction
is called {\it ordinary} if $a_0$ is integer and the rest elements
are positive integers. A finite continued fraction is {\it odd}
({\it even}) if it contains odd (even) elements.

\begin{proposition}
A rational number has a unique odd and a unique even continued
fractions.
\\
An irrational number has a unique infinite continued fraction.
\qed
\end{proposition}

We continue with several definitions of integer geometry. A point
is said to be {\it integer} if all its coefficients are integers.
The {\it integer length} of a segment $AB$ is the number of
integer points inside the segment plus one, denote it by
$\il(AB)$. The {\it integer sine} of the angle $ABC$ is the index
of the integer sublattice generated by the integer vectors of the
segments $BA$ and $BC$ in the whole lattice, we denote in by
$\isin (ABC)$. For more information on lattice (in particularly
integer) trigonometry we refer to~\cite{KarTrig} and
to~\cite{irrat}.

Let $C$ be a cone with vertex at the origin. Take the convex hull
of all integer points except the origin in $C$. The boundary of
the described convex hull is a broken line together with one or
two rays in case if there are some integer points on the edges of
the cone. The broken line in the boundary of the convex hull is
called the {\it sail} of $C$. (In some literature the sail is the
whole boundary of the convex hull, for us it is more convenient to
exclude the rays from the definition of a sail.)

Consider a positive real number $\alpha$. Denote by $C_\alpha$ the
cone with vertex at the origin and edges $\{(t,0)|t\ge 0\}$ and
$\{(t,\alpha t)|t\ge 0\}$. The sail for $C_\alpha$ is a finite
broken line if $\alpha$ is rational and one-side infinite broken
line if $\alpha$ is irrational.

So let the sail be a broken line $A_0A_1\ldots A_{n}$
($A_0A_1A_2\ldots$). Denote
$$
\begin{array}{l}
a_{2k-1}=\il(A_kA_{k+1}),\\
a_{2k}=\isin(A_{k-1}A_kA_{k+1})
\end{array}
$$
for all admissible $k$. The sequence $(a_0,a_1, \ldots, a_{2n})$
(or $(a_0,a_1,a_2, \ldots)$) is called the {\it lattice
length-sine sequence} for the cone $C_\alpha$ (or {\it
LLS-sequence for short}).

\begin{figure}
$$\epsfbox{7__5.1}$$
\caption{The broken line $A_0A_1A_2$ is the sail for
$C_{7/5}$.}\label{7/5}
\end{figure}

The connection of geometric and analytic properties of
LLS-sequence is introduced by the following theorem.
\begin{theorem}
Let $\alpha\ge 1$ and $(a_0,a_1, \ldots, a_{2n})$ $($or
$(a_0,a_1,a_2, \ldots)$$)$ be the  LLS-sequence for $C_\alpha$.
Then
$$
\alpha=[a_0,a_1,\ldots,a_{2n}] \quad (\hbox{or respectively
$\alpha=[a_0,a_1,\ldots]$}).
$$
\end{theorem}

\begin{proof}
This theorem is a reformulation of a Proposition~1.7.a
from~\cite{KarTrig} for finite continued fractions and
Theorem~2.7.a from~\cite{irrat} for the infinite continued
fractions. Finite case is also a particular case of
Corollary~\ref{geometry2}. So we skip the proof here.
\end{proof}

\section{Continued fractions with arbitrary coefficients}

In this section we generalize geometry of ordinary continued
fractions to the case of continued fractions with arbitrary
elements. We show a relation between odd or infinite continued
fractions and broken lines in the plane having a selected point
(say, the origin). We conclude this section with a few words about
conditions for a broken line to be closed in terms of elements of
the corresponding continued fraction.

Further we use the following notation. For a couple of vectors $v$
and $w$ denote by $|v\times w|$ the oriented volume of the
parallelogram spanned by the vectors $v$ and $w$.

\subsection{Construction of broken lines from the elements of
continued fractions}\label{algorithm} In this subsection we give a
natural geometric interpretation of an odd or infinite continued
fraction with arbitrary elements. It would be a broken line
defined by the positions of the first vertex and the selected
point $O$, direction of the first edge, and the continued
fraction.

So consider a continued fraction $[a_0,\ldots, a_{2n}]$. We are
also given by the vertex $A_0$, selected point $O$, and the
direction $v$ of the first edge. We construct all the rest
vertices $A_k$ inductively in $k$.

\vspace{1mm}

{\it Base of induction}. For the second vertex we take
$$
A_1=A_0+\lambda v,
$$
where $\lambda$ is defined from the equation $|OA_0\times
OA_1|=a_1$.

\vspace{1mm}

{\it Step of induction.} Suppose now we have the points $A_0,
\ldots, A_k$, for $k\ge 1$, Let us get $A_{k+1}$. Consider a point
$$
P=A_{k}+\frac{a_{2k-1}+1}{a_{2k-1}}A_{k-1}A_{k}.
$$
In other words $P$ is a point in the line $A_{k-1}A_k$ such that
the area $OA_{k}P$ equals 1. Let
$$
Q=P+a_{2k}OA_k.
$$
Finally the point $A_{k+1}$ is defined as follows (see on
Figure~\ref{induction})
$$
A_{k+1}=A_k+a_{2k+1}A_kQ.
$$

\begin{figure}
$$\epsfbox{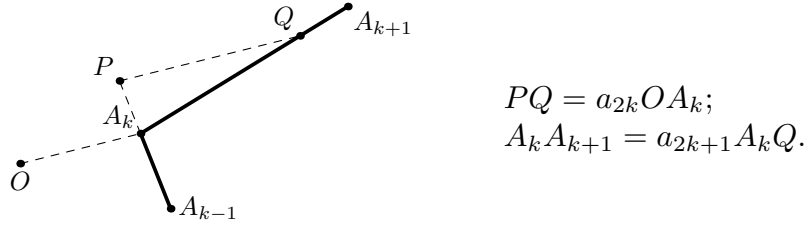}$$
\caption{Construction of $A_{k+1}$.}\label{induction}
\end{figure}

Let us now explain a geometric meaning of the elements of
continued fractions in terms of characteristics of the
corresponding broken line.

\begin{proposition}\label{properties}
The following holds
$$
\begin{array}{l}
a_{2k+1}=|OA_k\times OA_{k+1}|, \quad k=0,\ldots, n,\\
\displaystyle
a_{2k}=\frac{|A_{k}A_{k-1}\times A_{k}A_{k+1}|}{a_{2k-1}a_{2k+1}}, \quad k=1,\ldots, n.\\
\end{array}
$$
\end{proposition}

\begin{proof}
We prove this statement by induction in $k$.

\vspace{1mm}

{\it Base of induction.} From the definition of the point $A_1$ we
get
$$
|OA_0\times OA_1|=a_1.
$$

\vspace{1mm}

{\it Step of induction.} Let the statement be true for $k{-}1$, we
prove it for $k$.

First, we verify the formula for $a_{2k+1}$:
$$
|OA_k\times OA_{k+1}|=a_{2k+1}|OA_{k}\times
OQ|=a_{2k+1}|OA_k\times OP|= \frac{a_{2k+1}}{a_{2k-1}}
|OA_{k-1}\times OA_k|=a_{2k+1}.
$$
The last equality holds by induction.

Second, for $a_{2k}$ we have
$$
\begin{array}{l}
\displaystyle
\frac{|A_kA_{k-1}\times
A_kA_{k+1}|}{a_{2k-1}a_{2k+1}}= \frac{|A_kA_{k-1}\times
A_{k}Q|}{a_{2k-1}}=
|PA_k\times PQ|= a_{2k}|OA_k\times OP|
=
\\
\displaystyle \frac{a_{2k}}{a_{2k-1}}|OA_{k-1}\times OA_k|=a_{2k}.
\end{array}
$$
The step of induction is completed.
\end{proof}

\begin{example}
Let us construct a broken line having the first vector
$A_0=(1,0)$, the direction $v=(1,0)$, and the continued fraction
$[a,b,c]$. Then we have
$$
A_1=(1,a).
$$
Further we find the corresponding points $P$ and $Q$:
$$
P=(1,1+a), \qquad Q=(1+b,1+a+ab).
$$
Finally we get
$$
A_2=(1+bc,a+c+abc).
$$
\end{example}

\subsection{Inverse problem}

Now suppose we have a point $O$ and a broken line $A_0\ldots A_n$
such that for any $k$ the points $O$, $A_k$, and $A_{k+1}$ are not
in a line. Let us extend the definition of the LLS-sequence for
this data.

We use equalities of Proposition~\ref{properties} to define the
elements:
$$
\begin{array}{l}
a_{2k+1}=|OA_k\times OA_{k+1}|, \quad k=0,\ldots, n;\\
a_{2k}=\frac{\displaystyle |A_kA_{k-1}\times
A_{k}A_{k+1}|}{\displaystyle a_{2k-1}a_{2k+1}} ,
\quad k=1,\ldots, n.\\
\end{array}
$$
We call the sequence $(a_0,\ldots, a_{2n})$ the {\it LLS-sequence}
of the broken line with respect to the point $O$, and
$[a_0,\ldots, a_{2n}]$ --- the {\it corresponding continued
fraction}.

\begin{proposition}\label{conjugate}
Let $A_0\ldots A_n$  and $B_0\ldots B_n$ be two broken lines with
LLS-sequences $(a_0,\ldots,a_{2n})$ and $(b_0,\ldots, b_{2n})$
respectively. Suppose the first broken line is taken to the second
by some operator in $SL(2,\r)$ with determinant equals $\lambda$.
Then we have:
$$
\left\{
\begin{array}{l}
a_{2k}=\lambda b_{2k}, \quad k=1,\ldots, n \\
a_{2k+1}=\frac{1}{\lambda} b_{2k+1}, \quad k=0, \ldots, n
\end{array}
\right. .
$$
\end{proposition}

\begin{proof}
The volume of any parallelogram is multiplied by $\lambda$, then
the statement follows directly from formulas of
Proposition~\ref{properties}.
\end{proof}

\subsection{On geometric meaning of corresponding continued fractions}
For this subsection we fix the point $O$ to be at the origin.

Consider a continued fraction $[a_0,a_1,\ldots, a_{k}]$ as a
rational function in variables $a_0, \ldots , a_{k}$. This
rational function is a ration of two polynomials with non-negative
integer coefficients, denote them by $P_{k}$ and $Q_{k}$. Actually
the polynomials $P_{k}$ and $Q_k$ are uniquely defined by the
condition
$$
\frac{P_k(a_0,a_1,\ldots, a_{k})}{Q_k(a_0,a_1,\ldots,
a_{k})}=[a_0,a_1,\ldots, a_{k}]
$$
and the condition that the coefficients of both polynomials are
non-negative integer coefficients.

\begin{remark}\label{intcoef}
Notice that the last condition is equivalent to the condition that
the polynomial $P_k(a_0,a_1,\ldots, a_{k})$ contains a monomial
$a_0a_1\cdot \ldots \cdot a_k$ with unit coefficient.
\end{remark}

\begin{theorem}\label{geometry}
Let $A_0\ldots A_n$ be a broken line such that $A_0=(1,0)$, and
$A_1=(1,a_0)$ is collinear to the vector $(0,1)$. Suppose its
LLS-sequence is $(a_0,a_1,\ldots, a_{2n})$. Then
$$
A_n=\big(Q_{2n+1}(a_0,a_1,\ldots, a_{2n}),P_{2n+1}(a_0,a_1,\ldots,
a_{2n})\big).
$$
\end{theorem}

\begin{proof}
We prove this statement by induction in $n$.

\vspace{1mm}

{\it Base of induction.} If the broken line is a segment $A_0A_1$
with LLS-sequence $(a_0)$ then $A_1=(1,a_0)$.

\vspace{2mm}

{\it Step of induction.} Suppose the statement holds for all
broken lines with $k$ vertices, let us prove it for an arbitrary
broken line with $k+1$ vertex.

Consider a broken line $A_0\ldots A_k$ with LLS-sequence $(a_0,
\ldots, a_{2k})$. Let us apply a linear transformation with unit
determinant taking $A_1$ to $(1,0)$ and the line $A_2A_1$ to the
line $x=1$. This transformation is uniquely defined by all these
conditions, it is
$$
T=
\left(
\begin{array}{cc}
a_0a_1+1 & -a_1\\
-a_0 & 1 \\
\end{array}
\right).
$$

Denote the resulting broken line by $B_0B_1\ldots B_k$. By
Proposition~\ref{conjugate} all the elements of the LLS-sequence
for $B_0B_1\ldots B_k$ are the same. By the assumption of
induction we have
$$
B_{k}=\big(Q_{2k-1}(a_2,\ldots, a_{2k}),P_{2k-1}(a_2,\ldots,
a_{2k})\big).
$$
Denote the coordinates of $B_{k}$ by $q$ and $p$ respectively.
Then we have
$$
A_k=T^{-1}(B_k)=\big(p+a_1q, a_0p+(a_0a_1+1)q \big).
$$
The polynomials satisfy
$$
\frac{a_0p+(a_0a_1+1)q}{p+a_1q}=a_0+\frac{1}{a_1+\frac{\displaystyle
p}{\displaystyle q}}=
\frac{P_{2k+1}(a_0,a_1,\ldots, a_{2k})}{Q_{2k+1}(a_0,a_1,\ldots,
a_{2k})}.
$$
Notice that the polynomial $a_0p+(a_0a_1+1)q$ has a unit
coefficient in the monomial $a_0a_1\cdot\ldots \cdot a_{2k}$
coming from $a_0a_1q$. Therefore (see Remark~\ref{intcoef}),
$a_0p+(a_0a_1+1)q$ coincides with $P_{2k+1}(a_0,a_1,\ldots,
a_{2k})$ and $(p+a_1q)$ coincides with $Q_{2k+1}(a_0,a_1,\ldots,
a_{2k})$. So we are done with the step of induction. This
concludes the proof of the theorem.
\end{proof}

In particular we get the following corollary. In the classical
case it forms the basis of geometry of ordinary continued
fractions.

\begin{corollary}\label{geometry2}
Let $A_0\ldots A_n$ be a broken line such that $A_0=(1,0)$, and
$A_0=(1,a_0)$. Suppose that the corresponding continued fraction
is $\alpha=[a_0,a_1,\ldots, a_{2n}]$ and $A_n=(x,y)$. Then
$$
\frac{y}{x}=\alpha.
$$
$($If the corresponding continued fraction has an infinite value,
then $x/y=0$.$)$ \qed
\end{corollary}

This corollary implies the following statement.

\begin{corollary}
Let $A_0\ldots A_n$ and $B_0\ldots B_m$ be two broken lines with
$B_0=A_0$, such that the vector of the first edges either have the
same direction if $a_0/b_0>0$ or opposite otherwise. Suppose the
corresponding continued fractions coincide:
$$
[a_0,\ldots, a_{2n}]=[b_0, \ldots, b_{2m}].
$$
Then the points $A_n$, $B_m$, and the origin $O$ are in a line.
\end{corollary}

\begin{proof}
Consider an $SL(2,\r)$-operator taking $A_0$ to $(1,0)$ and $A_1$
to the line $x=1$. By Proposition~\ref{conjugate} the continued
fractions for both broken lines are not changed. Hence by
Corollary~\ref{geometry} the points $A_n$, $B_m$, and the origin
are in a line.
\end{proof}

\begin{remark}\label{cc}({\it On closed curves.}) How to
find that a certain continued fraction defines a closed curve?
From Theorem~\ref{geometry} we see that a broken line defined by
an LLS-sequence $(a_0,a_1,\ldots, a_{2n})$ with $A_0=(1,0)$ and
$A_0A_1$ being collinear to $(0,1)$ is closed if and only if
$$
Q_{2n+1}(a_0,a_1,\ldots, a_{2n})=1 \quad \hbox{and} \quad
P_{2n+1}(a_0,a_1,\ldots, a_{2n})=0.
$$
So, {\it these two polynomial conditions on the elements of the
LLS-sequence are necessary and sufficient conditions for the
broken line to be closed}.

Notice that the condition $P_{n+1}=0$ can be rewritten in the
following nice form
$$
[a_0,a_1,\ldots ,a_{2n}]=0.
$$
The condition $P_{n+1}=0$ was introduced in~\cite{KarTrig} for
certain broken-lines with integer vertices.
\end{remark}

\begin{example}
Let us study an example of broken lines consisting of three edges.
These curves are defined by continued fractions of type
$[a_0,a_1,a_2,a_3,a_4]$. Then the conditions for a broken line to
form a triangle are as follows:
$$
\left\{
\begin{array}{rcl}
a_0a_1a_2a_3a_4+a_0a_1a_2+a_0a_1a_4+a_0a_3a_4+a_2a_3a_4+a_0+a_2+a_4&=&0\\
a_1a_2a_3a_4+a_1a_2+a_1a_4+a_3a_4+1&=&1
\end{array}
\right. .
$$
(See on Figure~\ref{examples}.)

\begin{figure}
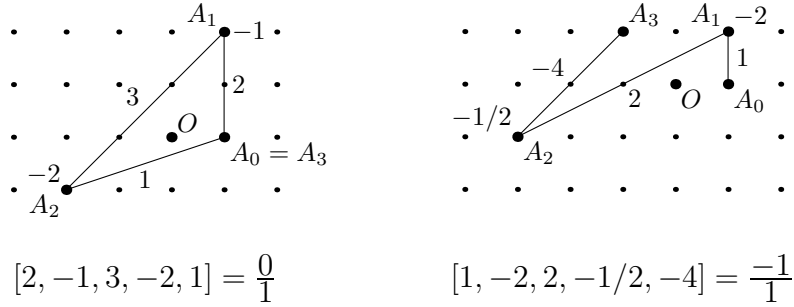

$$
\begin{array}{cc}
\epsfbox{examples.1}&\epsfbox{examples.2}\\
\hbox{$[ 2,-1,3,-2,1]$}=\frac{\displaystyle 0}{ \displaystyle 1}&
\hbox{$[ 1,-2,2,-1/2,-4]$}=\frac{\displaystyle -1}{ \displaystyle 1} \\
\end{array}
$$
\caption{Examples of broken lines and their continued
fractions.}\label{examples}
\end{figure}
\end{example}

There is one problem which is interesting in the frames of this
section. Suppose we have a broken line and two distinct points
$O_1$ and $O_2$. Then we have two LLS-sequences for the same curve
with respect to $O_1$ and $O_2$. {\it Study the conditions on the
initial data $($i.e.,~LLS-sequences, positions of the first points
of the broken lines, and direction of the first vector$)$ that
define congruent broken lines.}

\section{Differentiable curves}

Now let us study what happens if we consider a curve as a broken
line with infinitesimally small segments. It turns out that the
LLS-sequence ``splits'' to a couple of functions which we call
{\it areal and angular densities}. We introduce the necessary
notions and discuss basic properties of these functions. In
particular we show that the areal density is inverse to a velocity
of a point defined by the second Kepler law.

In this section we suppose that the curves has a natural (unit
length) parametrization.

\subsection{Definition of areal and angular densities}

Consider a curve $\gamma$ of class $C^2$ with an arc-length
parameter $t$. Let us define the areal and the angular elements at
a point similar to the discrete case.

\begin{definition}
The {\it areal density} and the {\it angular density} at $t$ are
respectively
$$
A(t)=\lim\limits_{\varepsilon \to 0} \frac{|O\gamma(t)\times
O\gamma(t+\varepsilon)|}{\varepsilon}=|O\gamma(t)\times
\dot{\gamma}(t)|
$$
and
$$
B(t)=\lim\limits_{\varepsilon \to 0}
\frac{|\gamma(t)\gamma(t-\varepsilon)\times
\gamma(t)\gamma(t+\varepsilon)|}{\varepsilon|O\gamma(t-\varepsilon)\times
O\gamma(t)||O\gamma(t)\times O\gamma(t+\varepsilon)|}.
$$
\end{definition}

Let us give geometric interpretations for the functions $A$ and
$B$. We start with $A$.

\begin{proposition}{\bf (Relation with the second Kepler law.)}
Suppose that a body moves by a trajectory of a curve $\gamma$ with
velocity $1/A$. Then the sector area velocity of a body is
constant and equals 1.
\end{proposition}

\begin{proof}
The proof follows directly from the definition.
\end{proof}

Instead of giving a geometrical interpretation of $B$, we prove
the following formula for $A^2B$. For a given curve $\gamma$
denote by $\kappa(t)$ the signed curvature at point $t$.
\begin{proposition}
Consider a point $\gamma(t)$ of a curve $\gamma$. Let the vectors
$O\gamma(t)$ and $\dot\gamma(t)$ be non-collinear. Then the
following holds.
$$
A^2(t)B(t)=\kappa(t).
$$
\end{proposition}

\begin{proof}
We have the following
$$
\begin{aligned}
A^2(t)B(t)&=\lim\limits_{\varepsilon \to 0}\left(
\left( \frac{|O\gamma(t)\times
O\gamma(t+\varepsilon)|}{\varepsilon}\right)^2
\frac{|\gamma(t)\gamma(t-\varepsilon)\times
\gamma(t)\gamma(t+\varepsilon)|}{\varepsilon|O\gamma(t-\varepsilon)\times
O\gamma(t)||O\gamma(t)\times O\gamma(t+\varepsilon)|} \right)
\\
&=\lim\limits_{\varepsilon \to 0}
\frac{|\gamma(t)\gamma(t-\varepsilon)\times
\gamma(t)\gamma(t+\varepsilon)|}{\varepsilon^3}.
\end{aligned}
$$
Notice that
$$
\begin{array}{l}
|\gamma(t)\gamma(t-\varepsilon)|=\varepsilon + o(\varepsilon),
\\
|\gamma(t)\gamma(t+\varepsilon)|=\varepsilon + o(\varepsilon),
\\
\sin (\gamma(t-\varepsilon)\gamma(t)\gamma(t+\varepsilon))=
\varepsilon\kappa(t) + o(\varepsilon).
\end{array}
$$
Therefore, for the volume of the corresponding parallelogram we
get
$$
|\gamma(t)\gamma(t-\varepsilon)\times
\gamma(t)\gamma(t+\varepsilon)|=\varepsilon^3\kappa(t) +
o(\varepsilon).
$$
Hence, $A^2(t)B(t)=\kappa(t)$.
\end{proof}

Now we prove the theorem on finite reconstruction of a curve
(i.e., in some small neighborhood) knowing the areal density and a
starting point. This is analogous to the algorithm that finds a
broken line by the elements of the corresponding continued
fraction described in Subsection~\ref{algorithm}. The significant
difference to the discrete case is that we do not need to know the
angular distribution function.

\begin{theorem}
Suppose we are given by the points $O$ and $\gamma(t_0)$ and the
areal density $A(t_0)$.

--- If $|A(t_0)|>|O\gamma(t_0)|$, then there is no finite curve with the
given data.

--- If $|O\gamma(t_0)|>|A(t_0)|>0$, then the curve is uniquely
defined in some neighborhood of the point $\gamma(t_0)$.
\end{theorem}

\begin{remark}
Notice that $A^2(t)B(t)$ defines the oriented curvature.
Therefore, if one knows the functions $A$ and $B$ then the curve
is uniquely reconstructed until the time $t_0$ where the vectors
$O\gamma(t_0)$ and $\dot \gamma(t_0)$ are collinear, or in other
words where $|A(t_0)|=0$.
\end{remark}

\begin{proof}
Consider a system of polar coordinates $(r,\varphi)$ with the
origin at the point $O$. To get the curve we should solve the
system of differential equations:
$$
\left\{
\begin{array}{l}
r^2\dot \varphi= A\\
\dot r^2+ r^2\dot \varphi^2  = 1\\
\end{array}
\right..
$$
This system is equivalent to the union of the following two
systems:
$$
\left\{
\begin{array}{l}
\dot \varphi= \frac{A}{r^2}\\
\dot r  = \sqrt{1-\frac{A^2}{r^2}}\\
\end{array}
\right.
\quad \hbox{and} \quad \left\{
\begin{array}{l}
\dot \varphi= \frac{A}{r^2}\\
\dot r  = -\sqrt{1-\frac{A^2}{r^2}}\\
\end{array}
\right..
$$

By the main theorem of theory of ordinary differential equations
(see for instance in~\cite{ODU}) this system has a finite solution
if $|r|>|A|>0$. This concludes the proof.
\end{proof}

Let us say a few words about density functions and their broken
line approximations. Let $\gamma(t)$ be a curve with arclength
parameter $t\in [0,T]$ and densities $A(t)$ and $B(t)$. For an
integer $n$ consider a broken line $\gamma_n=A_{0,n}\ldots
A_{n,n}$ such that $A_{i,n}=\gamma\big(\frac{i}{n}T\big)$. Let the
corresponding LLS-sequence be $(a_{0,n},\ldots,a_{2n,n})$. Denote
by $A_n$ and $B_n$ the following functions
$$
A_n(t)=a_{2\lfloor n t/T\rfloor+1,n}, \quad \hbox{and} \quad
B_n(t)=a_{2\lfloor nt/T\rfloor,n}.
$$

\begin{theorem}\label{convergence}
Let $\gamma$ be in $C^2$. Then the sequences of functions $(A_n)$
and $(B_n)$ pointwise converge to the functions $A$ and $B$
respectively.
\end{theorem}

\begin{proof}
This follows directly from the definition of density functions and
Proposition~\ref{properties}.
\end{proof}

It is interesting to investigate the inverse problem, it is still
open now, we formulate it in the last subsection.

\subsection{Example of curves and their continued fractions}

In this subsection we calculate the areal and angular densities
for straight lines, ellipses, and logarithmic spirals.

\begin{example}{\bf Lines.} Let us study the case of lines. Without lose
of generality we consider the point $O$ to be at the origin and
take the line $x=a$. Then the corresponding densities are
$$
A(t)=a \qquad \hbox{and} \qquad B(t)=0.
$$
\end{example}

\begin{example}{\bf Ellipses and their centers.}
Consider an ellipse $\frac{x^2}{a^2}+\frac{y^2}{b^2}=1$ with $a\ge
b>0$. Let $O$ be at the symmetry center of the ellipse i.e. at the
origin. Then the areal and angular densities are as follows
$$
A(t)=\frac{ab}{\sqrt{a^2\sin^2 t+ b^2\cos^2 t}} \qquad \hbox{and}
\qquad B(t)=\frac{1}{ab\sqrt{a^2\sin^2 t+ b^2\cos^2 t}}.
$$ Notice
that here we get the constant function for the ratio:
$$
\frac{A(t)}{B(t)}=a^2b^2.
$$
\end{example}

\vspace{1mm}

\begin{example}{\bf Ellipses and their foci.}
As in the previous example we consider an ellipse
$\frac{x^2}{a^2}+\frac{y^2}{b^2}=1$ with $a\ge b>0$. Let now $O$
be at one of the foci, for instance at $(-\sqrt{a^2-b^2},0)$. Then
the densities are as follows
$$
A(t)=\frac{ab+b\sqrt{a^2 -b^2} \cos t}{\sqrt{a^2\sin^2 t+
b^2\cos^2v t}} \quad \hbox{and} \quad
B(t)=\frac{a}{b\sqrt{a^2\sin^2 t +b^2 \cos^2 t} \big(a+ \cos t
\sqrt{a^2 -b^2} \big)^2}.
$$
\end{example}

{
\parindent=0cm
{\it Remark on the Kepler planetary motion.} If we put the Sun at
the chosen focus and consider a planet whose orbit is the ellipse,
then according to three Kepler laws the planet will move with
velocity $\lambda/A(t)$ at any $t$. Here the constant $\lambda$ is
defined from the third Kepler law: {\it the square of the orbital
period of a planet is directly proportional to the cube of the
semi-major axis of its orbit,} or in other words
$$
\frac{T^2}{a^3}=\frac{T^2_{e}}{a^3_e},
$$
where $T$ is the period for our orbit, and $T_e$ and $a_e$ are
respectively the period and the semi-major axis for the Earth.
Denote by $L$ the length of the ellipse, i.e.,
$$
L=4a\int\limits_{0}^{\pi/2}\sqrt{1-\left(1-\frac{b^2}{a^2}\right)\cos^2
t}\hbox{ } dt
$$
Since $T=|\lambda|\int_0^L |1/A(t)|dt$, we get
$$
\lambda=\pm \frac{T_{e}}{\int_0^L
|1/A(t)|dt}\left(\frac{a}{a_e}\right)^{\frac{3}{2}}.
$$
}

\vspace{1mm}

We skip a description for parabolas and hyperbolas, they are
similar to the case of ellipses.

\begin{example}{\bf Logarithmic spirals.}
Consider a logarithmic spiral
$$
\big\{\big(ae^{bt}\cos t,a e^{bt}\sin t\big)\big|t\in\r\big\}.
$$
Then the densities for this spiral are as follows
$$
A(t)=\frac{ae^{bt}}{\sqrt{b^2+1}} \quad \hbox{and} \quad
B(t)=\frac{e^{-3bt}\sqrt{b^2+1}}{a^3}.
$$
It is interesting to notice that for the spirals we have
$$
A^3(t)B(t)=\frac{1}{b^2+1},
$$
i.e., the products are constant functions.
\end{example}

Notice that if $A^2B$ is a constant function, then the curvature
is constant, and hence we get circles. What do we have if $AB$ (or
$A$) is constant?

\subsection{Open problems}

We conclude this section with two open problems concerning the
density functions. We start with a question on convergency that is
in some sense the inverse problem to Theorem~\ref{convergence}.

\begin{problem}
What properties should have the LLS-sequences of broken lines if
their sequence converges to certain curve.
\end{problem}

The second problem comes from Remark~\ref{cc} on closed broken
lines.
\begin{problem}
What are the conditions on the functions $A(t)$ and $B(t)$ for the
resulting curve $\gamma$ to be closed?
\end{problem}

\vspace{5mm}

\end{document}